\documentclass[a4paper,12pt,twoside]{article}
\input{preamble}

\title{Feedback vertex sets in oriented graphs\footnote{This research was partially funded by the French National Research Agency (ANR) under grant agreement No. ANR-24-CE48-3758-01. In accordance with the objective of open access dissemination, the author applies a Creative Commons Attribution (CC-BY) license to any accepted article or manuscript (AAM) resulting from this submission.}}
\author[1]{Simon Dreyer}
\affil[1]{LIRMM, Université de Montpellier, CNRS, Montpellier, France.}

\begin{document}
\maketitle

\begin{abstract}
  For an oriented graph $G$, denote by $\fv(G)$ the minimum number of vertices whose deletion from $G$ makes it acyclic. We show that an oriented graph $G$ on $n$ vertices and $m$ arcs satisfies $\fv(G) \le \frac{2n+m+h}{9}$ where $h$ denotes the number of connected components of $G$ that belong to a special class of oriented graphs. This result has three consequences. First, when $G$ is planar, we obtain that $\fv(G) \le \frac{2n+m}{9}$. In particular, this implies that $\fv(G) \le \frac{5n-6}{9}$ for any planar oriented graph $G$, improving the best known upper bound of $\frac{3n}{5}$~[Borodin, Discrete Mathematics, 1979]. Then, applying this inequality to the planar digraphs without directed triangles, we get that $\fv(G) \le \frac{6n-8}{13}$, which improves the current best bound of $\frac{n}{2}$~[Li and Mohar, SIAM Journal on Discrete Mathematics, 2017]. Finally, when $G$ has maximum degree 6, we have $\fv(G) \le \frac{4n}{7}$ and this bound is tight, answering a conjecture of Ai, Gutin, Liu, Yeo and Zhou~[arXiv:2512.01676, 2025].
\end{abstract}

\section{Introduction}

Given a graph $G=(V,E)$, oriented or not, a \emph{feedback vertex set} is a set $S\subseteq V$ of vertices of $G$ such that the subgraph induced by $V\setminus S$ contains no cycle; in the oriented case, this means that it contains no directed cycle. In this case, we also say that $V\setminus S$ is an \emph{induced acyclic subgraph}. Computing the \emph{minimum size of a feedback vertex set}, denoted by $\fv(G)$, is a hard problem even for planar graphs (see~\cite{Karp72, gareyjohnson1979}) and has been extensively studied in the literature.
 
The smallest size of a feedback vertex set of $G$ is called the \emph{decycling number} of $G$, and the highest order of an induced acyclic subgraph of $G$ is called the \emph{acyclic number} of $G$, denoted respectively by $\fv(G)$ and $a(G)$. Note that the sum of the decycling number and the acyclic number of $G$ is equal to the order of $G$ i.e. $|V(G)| = a(G) + \fv(G)$.

For planar graphs, by the Four Color Theorem, it is known that every planar graph on $n$ vertices contains an independent set of vertices of size $\frac{n}{4}$. A widely open conjecture proposes a strengthening of this result:

\begin{conjecture}[Albertson and Berman, 1976~\cite{AB76}]\label{conj:albertson-berman}
    Every simple undirected planar graph $G$ on $n$ vertices has $\fv(G) \le \frac{n}{2}$.
\end{conjecture}

Equivalently, \Cref{conj:albertson-berman} states that every undirected planar graph $G$ on $n$ vertices has an induced acyclic subgraph of at least $\frac{n}{2}$ vertices. The best known general upper bound for \Cref{conj:albertson-berman} is due to Borodin~\cite{B79}.

\begin{theorem}[Borodin, 1979~\cite{B79}]\label{thm:borodin}
    Every planar graph $G$ on $n$ vertices satisfies $\fv(G) \le \frac{3n}{5}$.
\end{theorem}

Better upper bounds are known depending on the \emph{girth} (the length of a smallest cycle)~\cite{DMP16,SX17,KL17}. The case of triangle-free planar graphs was recently particularly studied. Successive improvements of the upper bound for $\fv(G)$ were obtained showing that $\fv(G) \le \frac{m}{4}$ (Alon et al.~\cite{AMT01}), then $\fv(G) \le \frac{3n+6m}{32}$ (Salvatipour~\cite{S06}), then $\fv(G) \le \frac{6n+7m}{44}$ (Dross, Montassier and Pinlou~\cite{DMP19}) and finally $\fv(G) \le \frac{4n}{9}$ (Le~\cite{L18}) which is the best known upper bound. The conjectured upper bound for triangle-free planar graphs is $\fv(G) \le \frac{3n}{8}$ and was proposed by Akiyama and Watanabe~\cite{AW87}.

We switch to oriented graphs (or orgraphs) which have no parallel arcs, no digons and no loops. The length of a shortest directed cycle of an orgraph $G$ is called the \emph{digirth} of $G$.

The \Cref{thm:borodin} obtained by Borodin~\cite{B79} in 1979 directly imply that $\fv(G) \le \frac{3n}{5}$ for planar oriented graphs $G$ on $n$ vertices. Surprisingly, this is still the best known upper bound for $\fv(G)$ in planar oriented graphs.

Albertson posed the following weakening of \Cref{conj:albertson-berman} (see~\cite{Mohar2002, W06, KVW17}): does every planar orgraph contain an acyclic set of size at least $\frac{n}{2}$? In other words, is it true that $\fv(G) \le \frac{n}{2}$ for any planar oriented graph $G$ on $n$ vertices? If true, the conjecture would be tight since Knauer, Valicov and Wenger~\cite{KVW17} constructed an infinite family of planar orgraphs $G$ on $n$ vertices with $\fv(G) = \frac{n-1}{2}$.
Moreover, the question is a weakening of the following famous conjecture.

\begin{conjecture}[Neumann-Lara, 1985~\cite{NL85}]\label{conj:NL}
    Every planar oriented graph can be vertex-partitioned into two acyclic sets.
\end{conjecture}

While \Cref{conj:NL} remains open in general, it was established by Li and Mohar~\cite{LM17} for planar orgraphs of digirth at least 4. Therefore, $\fv(G) \le \frac{n}{2}$ for planar oriented graphs $G$ of digirth at least 4.
As in the undirected case, better upper bounds for $\fv(G)$ are known when the digirth $g$ is large. For $g \ge 6$, Esperet, Lemoine and Maffray~\cite{ELM17} proved that $\fv(G) \le \frac{2n-6}{g}$. This result was recently improved by the author together with Pinlou and Valicov~\cite{DPV26}, who showed that $\fv(G) \le \frac{n-2}{g-2}$ for planar orgraphs of digirth $g \ge 3$ and that this inequality is strict when $G$ is not a directed cycle. They also showed the following inequality that will be used in this article to prove one of the corollaries.
\begin{theorem}[Dreyer, Pinlou and Valicov, 2026~\cite{DPV26}]\label{thm:DPV26}
    Let $G$ be a strongly connected planar oriented graph of digirth $g \ge 4$ having $n$ vertices and $m$ arcs. Then $\fv(G) \le \frac{n-2}{g-3} - \frac{m}{g(g-3)}$.
\end{theorem}

Relaxing the planarity assumption, other parameters have been considered in the literature. Knauer, La and Valicov~\cite{KLV22} studied oriented graphs of bounded degeneracy or treewidth and proved that $\fv(G) \le \frac{k-1}{k+1}n$ for orgraphs $G$ of degeneracy $k$ and that this inequality is strict when $k$ is odd. They also showed that $\fv(G) \le \frac{k}{k+3}n$ for orgraphs $G$ of treewidth $k$. Ai, Gutin, Liu, Yeo and Zhou~\cite{AGLYZ25} studied the size of a feedback vertex set in oriented graphs having a bounded maximum degree.

\begin{theorem}[Ai, Gutin, Liu, Yeo and Zhou, 2025~\cite{AGLYZ25}]\label{thm:AGLYZ}
    Every oriented graph $G$ with maximum degree at most 4 satisfies $\fv(G) \le \frac{3n}{7}$ and every oriented graph $G$ with maximum degree at most 5 satisfies $\fv(G) \le \frac{n}{2}$.
\end{theorem}

They also posed the following conjecture for oriented graphs of maximum degree 6.

\begin{conjecture}[Ai, Gutin, Liu, Yeo and Zhou, 2025~\cite{AGLYZ25}]\label{conj:AGLYZ}
    Every oriented graph $G$ with maximum degree at most 6 satisfies $\fv(G) \le \frac{4n}{7}$.
\end{conjecture}

They noted that, if true, this inequality would be tight since the Paley graph $P^7$ of order 7 has maximum degree 6 and $\fv(P^7) = \frac{4n(P^7)}{7}$.

Some of the recent cited results were obtained by using the discharging method, minimal counterexamples and isolating certain specific class of subgraphs~\cite{AGLYZ25,L18}. Using these techniques, we provide an upper bound for $\fv(G)$ for oriented graphs which both answers the \Cref{conj:AGLYZ} and improves the oriented version of \Cref{thm:borodin}.

\section{Contributions}

Let $P^7$ be the Paley graph of order 7 (see \Cref{subfig:P7}). It is a 6-regular oriented graph with 7 vertices and 21 arcs with $\fv(P^7) = 4$. It also has the property that for any three vertices $x,y,z \in V(P^7)$ it is possible to extend $\{x,y,z \}$ to a minimal feedback vertex set of $P^7$. 

We then define $\mathcal{H}$ as the family of graphs containing $P^7$ as well as all graphs that can be obtained by taking two graphs from $\mathcal{H}$ and connecting them by an arc (see \Cref{subfig:graphInH} for an example). Thus every graph $H \in \mathcal{H}$ is a tree-like branching of copies of $P^7$ linked by arcs. If $H \in \mathcal{H}$ is a branching of $k$ copies of $P^7$ then $H$ has $7k$ vertices and $21k + (k-1) = 22k - 1$ arcs. Moreover, $\fv(H) = 4k$ and $a(H) = 3k$.

Let us also note that the minimal feedback vertex sets of $H$ are exactly the unions of minimal feedback vertex sets of each copy of $P^7$ that compose $H$. In particular we have the following fact.

\begin{fact}\label{fact:fvsGraphInH}
    For every graph $H \in \mathcal{H}$ and for any three vertices $x,y,z \in V(H)$ it is possible to extend $\{x,y,z \}$ to a minimal feedback vertex set of $H$.
\end{fact}

 \begin{figure}[ht]
    \centering
    \begin{subfigure}[ht]{0.48\textwidth}
        \centering
        \includegraphics[scale = 0.6]{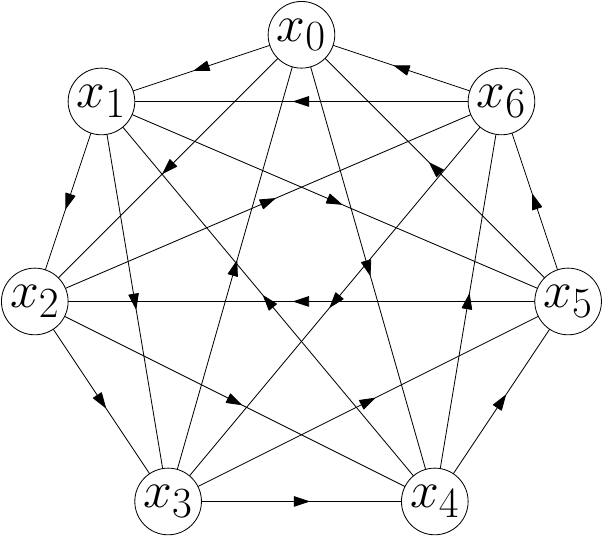}
        \caption{$P^7$: The Paley graph of order 7.}\label{subfig:P7}
    \end{subfigure}
    \begin{subfigure}[ht]{0.48\textwidth}
        \centering
        \includegraphics[scale = 0.45]{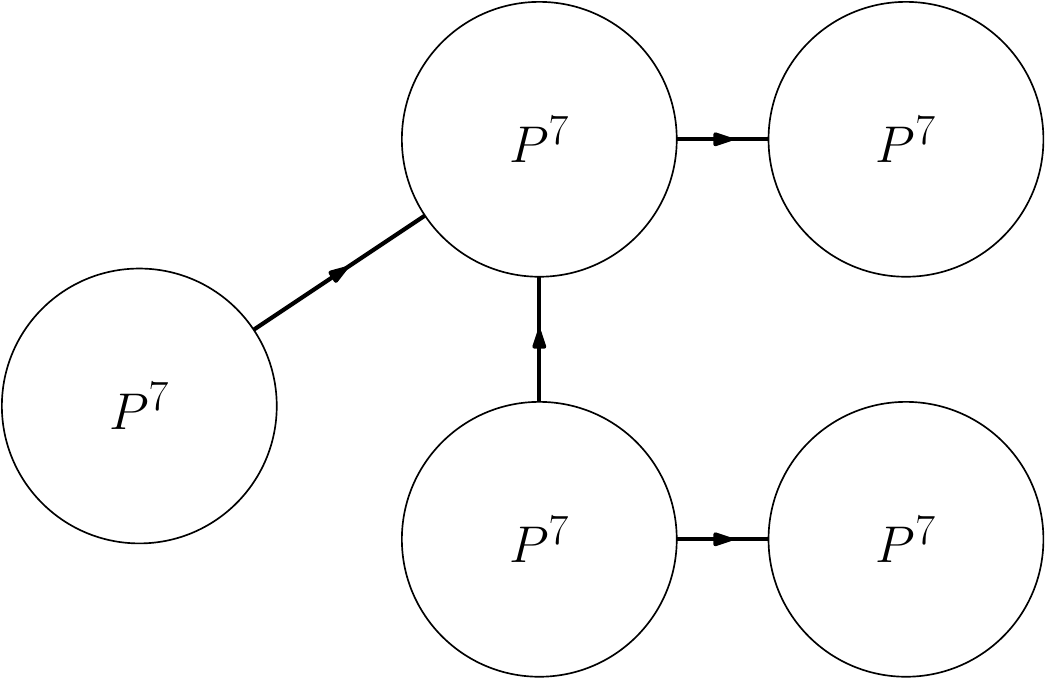}
        \caption{A graph in $\mathcal{H}$.}\label{subfig:graphInH}
    \end{subfigure}
    \caption{Illustration of the graph $P^7$ (left) and a graph in $\mathcal{H}$ (right).}
\end{figure}

For a planar oriented graph $G$ we denote by $n(G)$ the number of vertices, $m(G)$ the number of arcs and $h(G)$ the number of connected components of $G$ belonging to $\mathcal{H}$. When there is no ambiguity we simply write $n$, $m$ and $h$ instead of $n(G)$, $m(G)$ and $h(G)$.

\begin{theorem}\label{thm:main}
    Every oriented graph $G$ satisfies $a(G) \ge \frac{7n - m - h}{9}$.
\end{theorem}

Applied to planar oriented graphs, \Cref{thm:main} gives a better upper bound for $\fv(G)$ than the one due to Borodin~\cite{B79}.

\begin{corollary}\label{cor:planar}
    Every planar oriented graph $G$ satisfies $\fv(G) \le \frac{2n + m}{9}$ and thus $\fv(G) \le \frac{5n-6}{9}$.
\end{corollary}

\begin{proof}
    Let $G$ be a planar oriented graph on $n$ vertices and $m$ arcs. Since $P^7$ is not planar, we have $h(G) = 0$. Therefore, by \Cref{thm:main} we have:
    \begin{displaymath}
        a(G) \ge \frac{7n - m}{9} \quad \text{and then} \quad \fv(G) \le n - \frac{7n - m}{9} = \frac{2n + m}{9} \le \frac{5n-6}{9},
    \end{displaymath}
    where the last inequality follows from the fact that a planar oriented graph on $n$ vertices has at most $3n-6$ arcs.
\end{proof}

Restricted to planar oriented graphs having no oriented triangles, \Cref{thm:main} improves the upper bound of $\fv(G)$ given by Li and Mohar~\cite{LM17}.

\begin{corollary}\label{cor:planar_digirth4}
    Every planar oriented graph $G$ of digirth at least 4 satisfies $\fv(G) \le \frac{6n -8}{13}$.
\end{corollary}

\begin{proof}
    Suppose first that $G$ is strongly connected. Let $k = \fv(G)$. From \Cref{cor:planar} we have $m \ge 9k - 2n$. From \Cref{thm:DPV26} applied with $g = 4$ we have:
    \begin{displaymath}
        k \le n-2 - \frac{m}{4} \le n-2 - \frac{9k - 2n}{4} \quad \text{which gives} \quad k \le \frac{6n - 8}{13}.
    \end{displaymath}
    Now suppose that $G$ is not strongly connected. Let $G_1, \ldots, G_p$ be the strongly connected components of $G$. We have $\fv(G) = \sum_{i=1}^{p} \fv(G_i)$ and $n(G) = \sum_{i=1}^{p} n(G_i)$. Therefore, we can apply the previous inequality to each strongly connected component $G_i$ and we obtain:
    \begin{displaymath}
        \fv(G) = \sum_{i=1}^{p} \fv(G_i) \le \sum_{i=1}^{p} \frac{6n(G_i) - 8}{13} = \frac{6n(G) - 8p}{13} \le \frac{6n(G) - 8}{13}.
    \end{displaymath}
\end{proof}

Moreover, the formula given by \Cref{thm:main} answers the Conjecture 4.2 of~\cite{AGLYZ25} which states that $\fv(G) \le \frac{4n}{7}$ for graphs of maximum degree 6.

\begin{corollary}\label{cor:degreeMax6}
    Every oriented graph $G$ with maximum degree at most 6 satisfies $\fv(G) \le \frac{4n}{7}$.
\end{corollary}

\begin{proof}
    For a graph of maximum degree 6, we have $m \le 3n$ and as the smallest graph of $\mathcal{H}$ has 7 vertices, we have $h \le \frac{n}{7}$. Therefore, if $G$ has maximum degree 6, according to \Cref{thm:main} we have:
    \begin{displaymath}
        \fv(G) \le \frac{2n+m+h}{9} \le \frac{2n + 3n + \frac{n}{7}}{9} = \frac{4n}{7}.
    \end{displaymath}
\end{proof}

\Cref{cor:degreeMax6} is tight as the graph $P^7$ has maximum degree 6 and $\fv(P^7) = \frac{4n(P^7)}{7}$.

\section{Proof of \Cref{thm:main}}

% chktex-file 44

To prove \Cref{thm:main} we proceed by minimal counterexample. In what follows, $G$ denotes a minimal counterexample to \Cref{thm:main}. We will prove a series of lemmas to eliminate local configurations in $G$ and eventually reach a contradiction.

\subsection{Elimination of subgraphs isomorphic to a graph in $\mathcal{H}$}

\begin{lemma}\label{lem:degreeP7}
    Let $H$ be a subgraph of $G$ isomorphic to a graph in $\mathcal{H}$ that is maximal for the inclusion relation, that is, there does not exist a subgraph $H'$ of $G$ isomorphic to a graph in $\mathcal{H}$ and strictly containing $H$. Then $\deg(H) \ge 8$ where $\deg(H)$ denotes the number of arcs of $G$ having one endpoint in $H$ and one endpoint in $G-H$.
\end{lemma}

\begin{proof}
    Suppose that $\degin(H) \le 3$, that is, there are at most 3 arcs from $G-H$ to $H$. Let $k$ denote the number of copies of $P^7$ in $H$. We construct an acyclic set $A_H$ of $H$ of size $3k$ by imposing that the (at most 3) vertices having incoming arcs from $G-H$ into $H$ are not in $A_H$ (using \Cref{fact:fvsGraphInH}). Let $A'$ be a maximal acyclic set of $G' = G-H$. Then $A = A' \cup A_H$ is an acyclic set of $G$. We also note that $|E(G')| = m - (22k +1 + \deg(H))$ and $|V(G')| = n - 7k$. Moreover, we have $h(G') \le h  + \deg(H) - 1$ because if removing $H$ adds exactly $\deg(H)$ connected components to $G'$ then $H$ was not maximal for the inclusion relation (we can add these $\deg(H)$ components to $H$ to form a larger subgraph isomorphic to a graph in $\mathcal{H}$). Then since $G'$ is not a counterexample to \Cref{thm:main} we have
    \begin{displaymath}
            | A | = 3k + |A' | \ge 3k + \frac{7(n - 7k) - (m - (22k +1 + \deg(H))) - (h  + \deg(H) - 1)}{9} =\frac{7n - m - h}{9}.
    \end{displaymath}
    This contradicts the fact that $G$ is a counterexample to \Cref{thm:main} and we deduce that $\degin(H) \ge 4$. Similarly, we show that $\degout(H) \ge 4$ and consequently $\deg(H) \ge 8$.
\end{proof}

\begin{lemma}\label{lem:degree7}
    $G$ contains no vertex of degree 7 or more.
\end{lemma}

\begin{proof}
    Let $x$ be a vertex of degree at least 7. Let $G' = G - x$. By \Cref{lem:degreeP7}, $h(G') \le h + (\deg(x) - 7)$. Let $A$ be a maximal acyclic set of $G'$. It is also an acyclic set of $G$ and since $G'$ is not a counterexample to \Cref{thm:main} we have
    \begin{displaymath}
            | A | \ge \frac{7(n - 1) - (m - \deg(x)) - (h  + \deg(x) - 7)}{9} =\frac{7n - m - h}{9}.
    \end{displaymath}
    This contradicts the fact that $G$ is a counterexample to \Cref{thm:main} and we deduce that $G$ contains no vertex of degree 7 or more.
\end{proof}

\begin{lemma}\label{lem:hNul}
    $G$ has no subgraphs isomorphic to a graph in $\mathcal{H}$. In particular, for all $X \subseteq V(G)$, we have $h(G-X) = 0$.
\end{lemma}

\begin{proof}
    First we show that $h(G) = 0$. Suppose that $h(G) > 0$. Then there exists a connected component $H$ of $G$ isomorphic to a graph in $\mathcal{H}$. Since $G$ contains no vertex of degree 7 or more, this connected component is necessarily $P^7$. We know that $a(P^7) = 3$ and that $P^7$ has 7 vertices and 21 arcs. Then since $G-H$ is not a counterexample to \Cref{thm:main} we have
    \begin{displaymath}
        a(G) = a(H) + a(G-H) \ge 3 + \frac{7(n - 7) - (m - 21) - (h  - 1)}{9} =\frac{7n - m - h}{9}.
    \end{displaymath}
    This contradicts the fact that $G$ is a counterexample to \Cref{thm:main} and we deduce that $h(G) = 0$.

    Now suppose there exists a set of vertices $X$ such that $h(G-X) > 0$. Then there exists a connected component $H$ of $G-X$ isomorphic to a graph in $\mathcal{H}$. Since $h(G) = 0$, $H$ was not a connected component of $G$. This means that there exists a vertex of $H$ that is connected to a vertex of $G-H$ by an arc. However, all vertices of $H$ have degree at least 6, so there exists a vertex of $H$ of degree at least 7 in $G$, which contradicts \Cref{lem:degree7}.
\end{proof}

\subsection{Elimination of degree less than 3 vertices}

The following proposition will be used repeatedly in the next lemmas to eliminate local configurations in $G$. It gives a sufficient condition to show that $G$ is not a minimal counterexample to \Cref{thm:main} by removing a set of vertices and adding some of them to an acyclic set of the remaining graph.

\begin{proposition}\label{prop:reduction}
    Let $(\alpha, \beta, \gamma)$ be a triple of integers satisfying $7\alpha - \beta = 9 \gamma$. If there exists $X \subset V(G)$ such that: 
    \begin{itemize}
        \item $X$ is a set of $\alpha$ vertices, i.e., $|X| = \alpha > 0$,
        \item Removing $X$ from $G$ deletes at most $\beta$ arcs, i.e., $\beta \le |E(G)| - |E(G-X)|$,
        \item For every acyclic set $A'$ of $G' = G-X$, one can find an acyclic set $A$ of $G$ with $|A| \ge |A'| + \gamma$.
    \end{itemize}
    Then $G$ is not a minimal counterexample to \Cref{thm:main}.
\end{proposition}

\begin{proof}
    Let $G$ be a minimal counterexample to \Cref{thm:main}. Take a triple $(\alpha, \beta, \gamma)$ such that $7\alpha - \beta = 9 \gamma$ and suppose there exists a set $X \subset V(G)$ satisfying the conditions of the proposition for this triple. By minimality of $G$ and \Cref{lem:hNul}, we have $a(G-X) \ge \frac{7(n-\alpha) - (m-\beta)}{9}$. Let $A$ be an acyclic set of $G-X$ of size at least $\frac{7(n-\alpha) - (m-\beta)}{9}$. By hypothesis there exists an acyclic set $A'$ of $G$ with $|A'| \ge |A| + \gamma$. Combining the inequalities yields
    \begin{displaymath}
        a(G) \ge |A'| \ge |A| + \gamma \ge \frac{7(n-\alpha) - (m-\beta)}{9} + \gamma = \frac{7n - m + (9\gamma - 7\alpha + \beta)}{9} = \frac{7n - m}{9}.
    \end{displaymath}
    This contradicts the fact that $G$ is a counterexample (recall that $h(G) = 0$ from \Cref{lem:hNul}).
\end{proof}

The \Cref{tab:triples} lists the triples $(\alpha, \beta, \gamma)$ that will be used in the next lemmas to eliminate local configurations in $G$ with \Cref{prop:reduction}.

\rowcolors{2}{gray!15}{white}
\begin{table}[ht]
    \centering
    \begin{tabular}{|c|c|c|}
        \hline
        $\alpha$ & $\beta$ & $\gamma$ \\
        \hline
        1 & -2 & 1 \\
        2 & 5 & 1 \\
        3 & 3 & 2 \\
        3 & 12 & 1 \\
        4 & 10 & 2 \\
        5 & 8 & 3 \\
        5 & 17 & 2 \\
        6 & 15 & 3 \\
        7 & 13 & 4 \\
        7 & 22 & 3 \\
        \hline
    \end{tabular}
    \caption{Triples $(\alpha, \beta, \gamma)$ satisfying $7\alpha - \beta = 9\gamma$ used in the proof of \Cref{thm:main}.}\label{tab:triples}
\end{table}

\begin{lemma}\label{lem:degree1}
    Every vertex of $G$ has both in-degree and out-degree at least 1. In particular, $G$ has no vertex of degree 0 or 1.
\end{lemma}

\begin{proof}
    Let $x\in V(G)$ with $\degout(x)=0$. Removing $x$ deletes $\deg(x)\ge0>-2$ arcs of $G$. Moreover, for every acyclic set $A'$ of $G'=G-x$, $A=A'+x$ is acyclic in $G$. Applying \Cref{prop:reduction} with the triple $(1,-2,1)$ shows that $G$ is not a minimal counterexample. The same argument applies if $\degin(x)=0$.
\end{proof}

\begin{lemma}\label{lem:degree2}
    $G$ contains no vertex of degree 2.
\end{lemma}

\begin{proof}
    Let $x\in V(G)$ be a vertex of degree 2. By \Cref{lem:degree1}, $x$ has one in-neighbor $u$ and one out-neighbor $v$.

    If $\degout(u)=1$ then removing $x,u,v$ deletes at least 3 arcs. Moreover, for every acyclic set $A'$ of $G'=G-x-u-v$, $A=A'+u+x$ is acyclic in $G$. Applying \Cref{prop:reduction} with the triple $(3,3,2)$ gives a contradiction (see \Cref{subfig:degree2a}).

    If $\degin(u)=1$, let $w$ be the in-neighbor of $u$. Removing $x,u,w$ deletes at least 3 arcs. For every acyclic set $A'$ of $G'=G-x-u-w$, $A=A'+u+x$ is acyclic in $G$. Again \Cref{prop:reduction} with $(3,3,2)$ yields a contradiction (see \Cref{subfig:degree2b}).

    Hence $\deg(u)=\degin(u)+\degout(u)\ge4$. Then removing $x,u$ deletes at least 5 arcs. For every acyclic set $A'$ of $G'=G-x-u$, $A=A'+x$ is acyclic in $G$. Applying \Cref{prop:reduction} with $(2,5,1)$ contradicts minimality (see \Cref{subfig:degree2c}).
\end{proof}

\begin{figure}[ht]
    \centering
    \begin{subfigure}[ht]{0.32\textwidth}
        \centering
        \includegraphics[scale = 0.6]{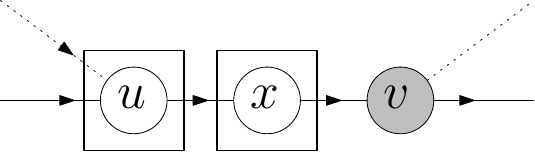}
        \caption{Case $\degout(u)=1$.}\label{subfig:degree2a}
    \end{subfigure}
    \begin{subfigure}[ht]{0.32\textwidth}
        \centering
        \includegraphics[scale = 0.6]{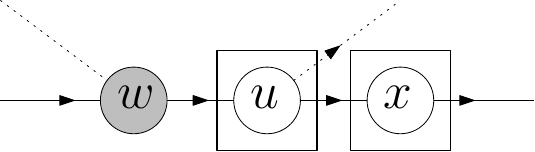}
        \caption{Case $\degin(u)=1$.}\label{subfig:degree2b}
    \end{subfigure}
    \begin{subfigure}[ht]{0.32\textwidth}
        \centering
        \includegraphics[scale = 0.6]{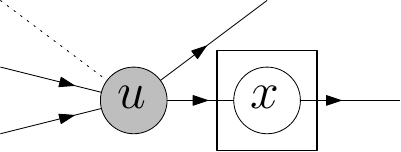}
        \caption{Case $\deg(u)\ge4$.}\label{subfig:degree2c}
    \end{subfigure}
    \caption{Illustration of the proof of \Cref{lem:degree2}. Removed vertices that are not added in the acyclic set are colored in gray, vertices added to the acyclic set are boxed.}\label{fig:degree2}
\end{figure}

\begin{lemma}\label{lem:degree3}
    Every vertex of $G$ has in-degree and out-degree at least 2. In particular, $G$ contains no vertex of degree 3.
\end{lemma}

\begin{proof}
    Let $x\in V(G)$ with $\degout(x)=1$. Let $u$ be the out-neighbor of $x$. By \Cref{lem:degree2}, $\deg(u),\deg(x)\ge3$, so removing $x,u$ deletes at least 5 arcs. For every acyclic set $A'$ of $G'=G-x-u$, $A=A'+x$ is acyclic in $G$. Applying \Cref{prop:reduction} with $(2,5,1)$ contradicts minimality (see \Cref{fig:degree3}).
\end{proof}

\begin{figure}[ht]
    \centering
    \includegraphics[scale = 0.6]{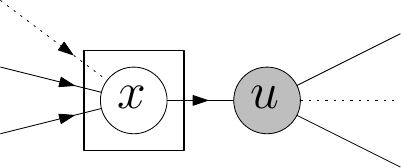}
    \caption{Illustration of the proof of Lemma~\ref{lem:degree3}.}\label{fig:degree3}
\end{figure}

\subsection{Elimination of degree 5 vertices}

\begin{lemma}\label{lem:2OutNeighbors}
    Let $x \in V(G)$ such that $\degout(x) = 2$ (respectively $\degin(x) = 2$). Let $u,v$ be the two out-neighbors (respectively in-neighbors) of $x$. We have $\deg(x) + \deg(u) + \deg(v) \le 14$. Moreover, if $\deg(x) + \deg(u) + \deg(v) = 14$ then $u$ and $v$ are adjacent.
\end{lemma}

\begin{proof}
    Suppose that $\deg(x) + \deg(u) + \deg(v) \ge 15$ or that $\deg(x) + \deg(u) + \deg(v) = 14$ with $u$ and $v$ non-adjacent. Under these hypotheses, removing $x,u,v$ deletes at least 12 arcs. Moreover, for every acyclic set $A'$ of $G' = G-x-u-v$, $A = A' + x$ is an acyclic set in $G$. By applying \Cref{prop:reduction} with the triple $(3,12,1)$, we find that $G$ is not a minimal counterexample to \Cref{thm:main} (see \Cref{fig:2voisins_sortant}).
\end{proof}

\begin{figure}[ht]
    \centering
    \includegraphics[scale = 0.6]{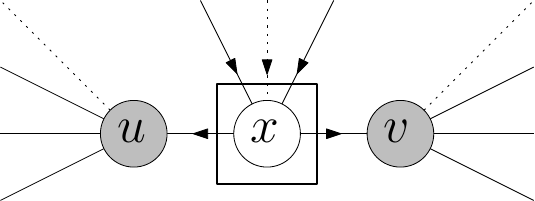}
    \caption{Illustration of the proof of Lemma~\ref{lem:2OutNeighbors}.}\label{fig:2voisins_sortant}
\end{figure}

\begin{lemma}\label{lem:degree5}
    $G$ contains no vertex of degree 5.
\end{lemma}

\begin{proof}
    Let $x \in V$ be a vertex of degree 5. Without loss of generality, assume that $\degout(x) = 2$. Let $u,v$ be the two out-neighbors of $x$.

    \begin{claim}\label{claim:degree5claim1}
        $\deg(u) = \deg(v) = 4$ and $u$ and $v$ are not adjacent.
    \end{claim}

    \begin{claimproof}[Proof of \Cref{claim:degree5claim1}]
        By \Cref{lem:2OutNeighbors} applied to $x$ (of degree 5), we have $\deg(u) + \deg(v) \le 9$ and if $\deg(u) + \deg(v) = 9$ then $u$ and $v$ are adjacent. It suffices therefore to show that $u$ and $v$ are not adjacent and we will deduce that $\deg(u) = \deg(v) = 4$. Without loss of generality, assume that $uv \in E$.
        
        If $\degout(u) = 2$ then let $w$ be the second out-neighbor of $u$ and by removing $x,u,v,w$ we delete at least 10 arcs. Moreover, for every acyclic set $A'$ of $G' = G-x-u-v-w$, $A = A' + u + x$ is an acyclic set in $G$. By applying \Cref{prop:reduction} with the triple $(4, 10, 2)$, we find that $G$ is not a minimal counterexample to \Cref{thm:main} (see \Cref{subfig:degree5claim1a}).
        
        Therefore $\degout(u) \ge 3$ and since $\deg(u) + \deg(v) \le 9$ this implies that $\deg(u) \ge 5$ and $\deg(v) = 4$. In particular $\degin(u) = 2$. Let $w$ be the second in-neighbor of $u$. By removing $x,u,v,w$ we delete at least 10 arcs. Moreover, for every acyclic set $A'$ of $G' = G-x-u-v-w$, $A = A' + u + x$ is an acyclic set in $G$. By applying \Cref{prop:reduction} with the triple $(4, 10, 2)$, we find that $G$ is not a minimal counterexample to \Cref{thm:main} (see \Cref{subfig:degree5claim1b}).
    \end{claimproof}

    \begin{figure}[ht]
        \centering
        \begin{subfigure}[ht]{0.45\textwidth}
            \centering
            \includegraphics[scale = 0.6]{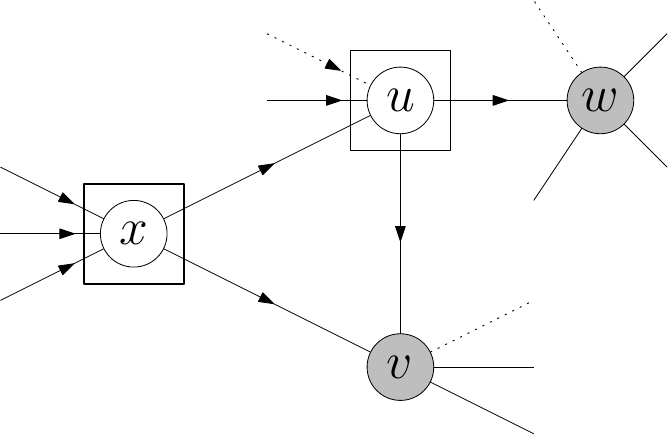}
            \caption{Case where $\degout(u) = 2$.}\label{subfig:degree5claim1a}
        \end{subfigure}
        \begin{subfigure}[ht]{0.45\textwidth}
            \centering
            \includegraphics[scale = 0.6]{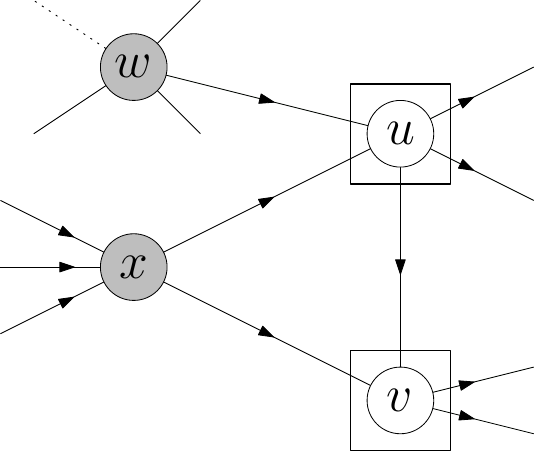}
            \caption{Case where $\degin(u) = 2$.}\label{subfig:degree5claim1b}
        \end{subfigure}
        \caption{Illustration of the proof of \Cref{claim:degree5claim1}.}\label{fig:degree5claim1}
    \end{figure}

    Let $y$ (respectively $z$) be the other in-neighbor of $u$ (respectively $v$).

    \begin{claim}\label{claim:degree5claim2}
        $y \neq z$, $\degin(y) = 2$ and $\degin(z) = 2$. Moreover, if $yx \in E$ (respectively $zx \in E$) then $\degout(y) = 3$ (respectively $\degout(z) = 3$).
    \end{claim}

    \begin{claimproof}[Proof of \Cref{claim:degree5claim2}]
        If $y = z$ then by removing $x,u,v,y$ we delete at least 10 arcs. Moreover, for every acyclic set $A'$ of $G' = G-x-u-v-y$, $A = A' + u + v$ is an acyclic set in $G$. By applying \Cref{prop:reduction} with the triple $(4, 10, 2)$, we find that $G$ is not a minimal counterexample to \Cref{thm:main} (see \Cref{subfig:degree5claim2a}). Therefore $y \neq z$.

        Suppose that $\degin(y) = 3$ so $\deg(y) \ge 5$. Then by \Cref{lem:2OutNeighbors} applied to $u$ we have $\deg(y) = 5$ and $yx \in E$. By removing $x,u,v,y$ we delete at least 10 arcs. Moreover, for every acyclic set $A'$ of $G' = G-x-u-v-y$, $A = A' + x + y$ is an acyclic set in $G$. By applying \Cref{prop:reduction} with the triple $(4, 10, 2)$, we find that $G$ is not a minimal counterexample to \Cref{thm:main} (see \Cref{subfig:degree5claim2b}). Therefore $\degin(y) = 2$ and similarly $\degin(z) = 2$.

        Suppose that $yx \in E$ and $\degout(y)=2$. By removing $x,u,v,y$ we delete at least 10 arcs. Moreover, for every acyclic set $A'$ of $G' = G-x-u-v-y$, $A = A' + x + y$ is an acyclic set in $G$. By applying \Cref{prop:reduction} with the triple $(4, 10, 2)$, we find that $G$ is not a minimal counterexample to \Cref{thm:main} (see \Cref{subfig:degree5claim2c}). Therefore $yx \in E \Longrightarrow \degout(y) = 3$. Similarly, $zx \in E \Longrightarrow \degout(z) = 3$.
    \end{claimproof}

    \begin{figure}[ht]
        \centering
        \begin{subfigure}[ht]{0.32\textwidth}
            \centering
            \includegraphics[scale = 0.5]{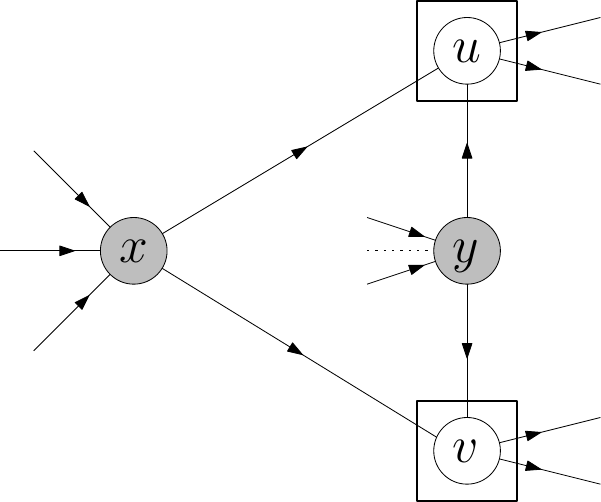}
            \caption{Case where $y=z$.}\label{subfig:degree5claim2a}
        \end{subfigure}
        \begin{subfigure}[ht]{0.32\textwidth}
            \centering
            \includegraphics[scale = 0.5]{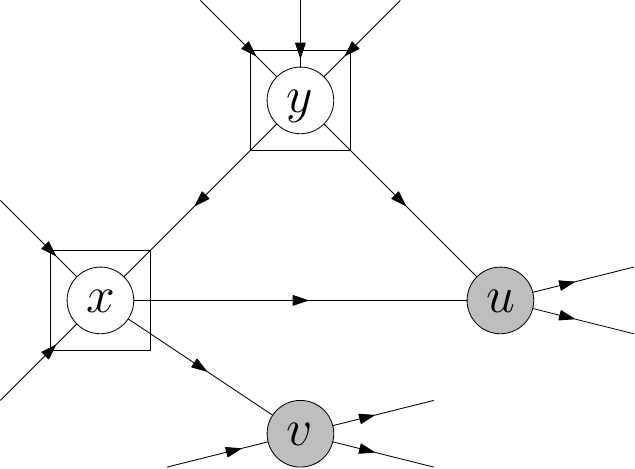}
            \caption{Case where $\degin(y) = 3$.}\label{subfig:degree5claim2b}
        \end{subfigure}
        \begin{subfigure}[ht]{0.33\textwidth}
            \centering
            \includegraphics[scale = 0.5]{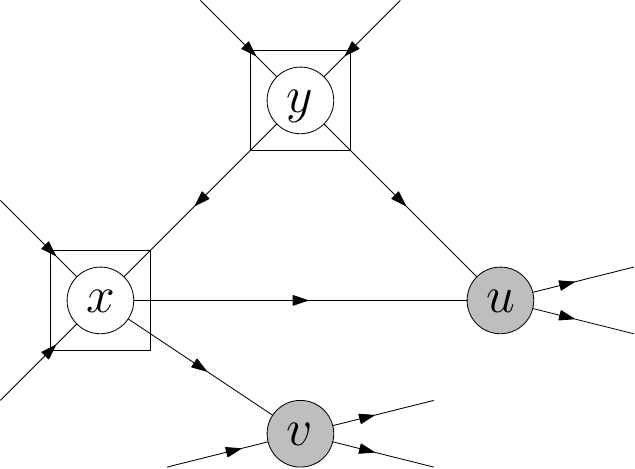}
            \caption{Case where $yx \in E$ and $\degout(y) = 2$.}\label{subfig:degree5claim2c}
        \end{subfigure}
        \caption{Illustration of the proof of \Cref{claim:degree5claim2}.}\label{fig:degree5claim2}
    \end{figure}

    \begin{claim}\label{claim:degree5claim3}
        $y$ and $z$ are not adjacent and moreover $uz, vy \notin E$.
    \end{claim}

    \begin{claimproof}[Proof of \Cref{claim:degree5claim3}]
        Suppose that $yz \in E$ and $uz \in E$. By removing $x,u,v,y,z$ we delete at least 8 arcs. Moreover, for every acyclic set $A'$ of $G' = G-x-u-v-y-z$, $A = A' + u + v + z$ is an acyclic set in $G$. By applying \Cref{prop:reduction} with the triple $(5, 8, 3)$, we find that $G$ is not a minimal counterexample to \Cref{thm:main} (see \Cref{subfig:degree5claim3a}).

        Suppose that $yz \in E$ and $uz \notin E$. Let $w$ be the second in-neighbor of $z$. From what precedes, $w$ is a vertex distinct from all the other vertices already cited. By removing $x,u,v,y,z,w$ we delete at least 15 arcs. Indeed we have 6 arcs among $x,u,v,y,z,w$, 3 entering $x$, 2 leaving $u$ and 2 leaving $v$ (since $ux, vx \notin E$), one of the two entering $y$ (we can have $vy \in E$ but $uy \notin E$) as well as one leaving $z$ (we have $zy \notin E$ and if $zx \in E$ then $\degout(z) = 3$ by the previous claim). Moreover, for every acyclic set $A'$ of $G' = G-x-u-v-y-z-w$, $A = A' + u + v + z$ is an acyclic set in $G$. By applying \Cref{prop:reduction} with the triple $(6, 15, 3)$, we find that $G$ is not a minimal counterexample to \Cref{thm:main} (see \Cref{subfig:degree5claim3b}). Therefore $yz \notin E$ and similarly $zy \notin E$.

        Suppose that $uz \in E$. Let $w$ be the second in-neighbor of $z$. From what precedes, $w$ is a vertex distinct from all the other vertices already cited. By removing $x,u,v,y,z,w$ we delete at least 15 arcs. Indeed we have 6 arcs among $x,u,v,y,z,w$, 3 entering $x$, 1 leaving $u$ and 2 leaving $v$ (since $ux, vx \notin E$), one of the two entering $y$ (we can have $vy \in E$ but $uy, zy \notin E$) as well as one leaving $y$ and $z$ (we have $yz \notin E$ (respectively $zy \notin E$) and if $yx \in E$ (respectively $zx \in E$) then $\degout(y) = 3$ (respectively $\degout(z) = 3$) by the previous claim). Moreover, for every acyclic set $A'$ of $G' = G-x-u-v-y-z-w$, $A = A' + u + v + z$ is an acyclic set in $G$. By applying \Cref{prop:reduction} with the triple $(6, 15, 3)$, we find that $G$ is not a minimal counterexample to \Cref{thm:main} (see \Cref{subfig:degree5claim3c}). Therefore $uz \notin E$ and similarly $vy \notin E$.
    \end{claimproof}

    \begin{figure}[ht]
        \centering
        \begin{subfigure}[ht]{0.32\textwidth}
            \centering
            \includegraphics[scale = 0.45]{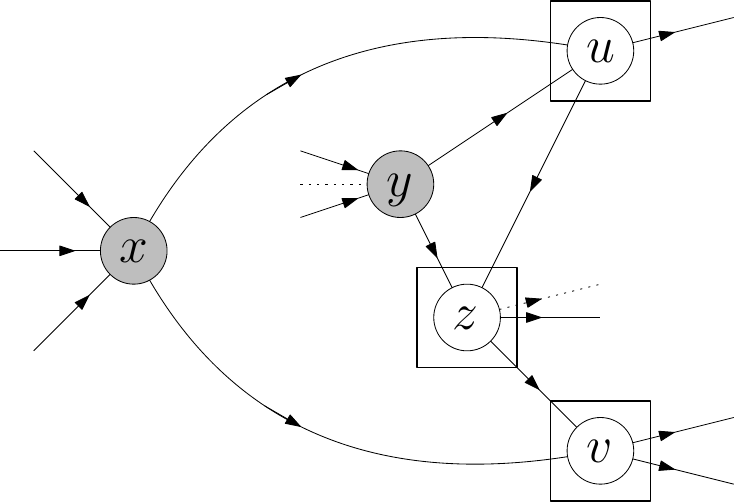}
            \caption{Case where $yz, uz \in E$.}\label{subfig:degree5claim3a}
        \end{subfigure}
        \begin{subfigure}[ht]{0.32\textwidth}
            \centering
            \includegraphics[scale = 0.45]{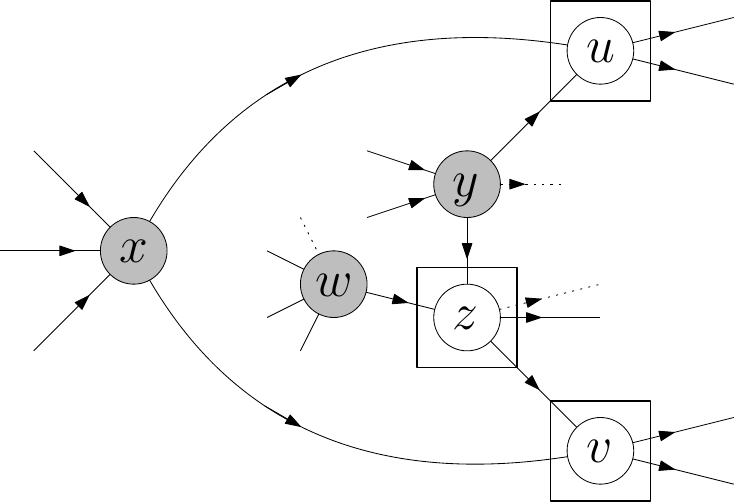}
            \caption{Case where $yz \in E$ and $uz \notin E$.}\label{subfig:degree5claim3b}
        \end{subfigure}
        \begin{subfigure}[ht]{0.33\textwidth}
            \centering
            \includegraphics[scale = 0.45]{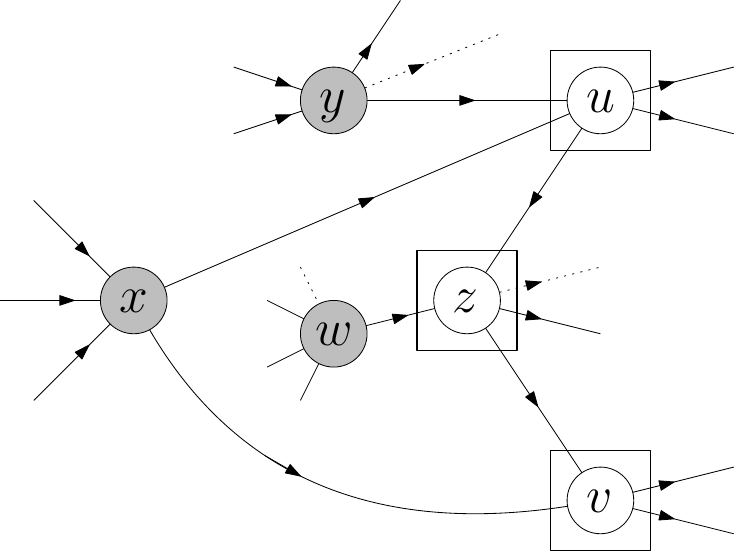}
            \caption{Case where $uz \in E$.}\label{subfig:degree5claim3c}
        \end{subfigure}
        \caption{Illustration of the proof of \Cref{claim:degree5claim3}.}\label{fig:degree5claim3}
    \end{figure}

    From the previous claims we know that $yz, zy, uz, vy \notin E$ and that if $yx \in E$ (respectively $zx \in E$) then $\degout(y) = 3$ (respectively $\degout(z) = 3$). Thus by removing $x,u,v,y,z$ we delete at least 17 arcs. Moreover, for every acyclic set $A'$ of $G' = G-x-u-v-y-z$, $A = A' + u + v$ is an acyclic set in $G$. By applying \Cref{prop:reduction} with the triple $(5, 17, 2)$, we find that $G$ is not a minimal counterexample to \Cref{thm:main} (see \Cref{fig:degree5}).
\end{proof}

\begin{figure}[ht]
    \centering
    \includegraphics[scale = 0.6]{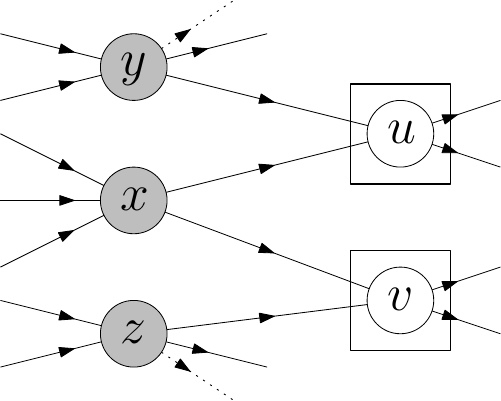}
    \caption{Illustration of the end of the proof of Lemma~\ref{lem:degree5}.}\label{fig:degree5}
\end{figure}

\subsection{Elimination of degree 6 vertices}

\begin{lemma}\label{lem:degree6balanced}
    Vertices of degree 6 have out-degree and in-degree equal to 3.
\end{lemma}

\begin{proof}
    Let $x$ be a vertex of degree 6 such that $\degout(x) = 2$. Let $u$ and $v$ be its two out-neighbors. By \Cref{lem:2OutNeighbors} applied to $x$, we must have $\deg(u) = \deg(v) = 4$ and $u$ and $v$ adjacent. Without loss of generality, assume $uv \in E$. Let $y$ be the second in-neighbor of $u$. By removing $x,u,v,y$ we lose at least 10 arcs. Moreover, for every acyclic set $A'$ of $G' = G-x-u-v-y$, $A = A' + u + v$ is an acyclic set in $G$. By applying \Cref{prop:reduction} with the triple $(4, 10, 2)$, we find that $G$ is not a minimal counterexample to \Cref{thm:main} (see \Cref{fig:degree6balanced}).
\end{proof}

\begin{figure}[ht]
    \centering
    \includegraphics[scale = 0.6]{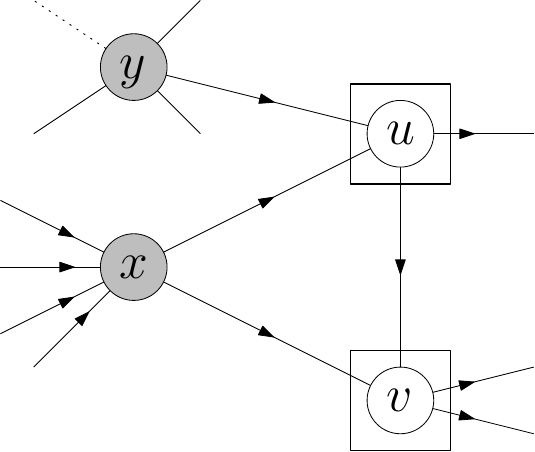}
    \caption{Illustration of the proof of Lemma~\ref{lem:degree6balanced}.}\label{fig:degree6balanced}        
\end{figure}

\begin{lemma}\label{lem:6arc4}
    There is no arc between a vertex of degree 6 and a vertex of degree 4.
\end{lemma}

\begin{proof}
    Begin the proof with the following claim that will be used several times in the rest of the proof of \Cref{lem:6arc4}.
    \begin{claim}\label{claim:6arc4claim}
        Let $x$ be a vertex of degree 6 and $a$ a vertex of degree 4. If $a$ is an in-neighbor (respectively out-neighbor) of $x$ then there exists a vertex $a'$ of degree 4 that is an out-neighbor (respectively in-neighbor) of both $x$ and $a$.
    \end{claim}

    \begin{claimproof}[Proof of \Cref{claim:6arc4claim}]
        Let $a'$ be the second out-neighbor of $a$. By \Cref{lem:2OutNeighbors} applied to $a$, we have $\deg(a') = 4$ and $a'$ is adjacent to $x$. Suppose that $a' x \in E$. Let $y$ be the second out-neighbor of $a'$. By removing $x,a,a',y$ we lose at least 10 arcs. Moreover, for every acyclic set $A'$ of $G' = G-x-a-a'-y$, $A = A' + a + a'$ is an acyclic set in $G$. By applying \Cref{prop:reduction} with the triple $(4, 10, 2)$, we find that $G$ is not a minimal counterexample to \Cref{thm:main} (see \Cref{fig:liaison_degre6_degre4}). Therefore $x a' \in E$ and $a'$ is an out-neighbor of both $x$ and $a$.
    \end{claimproof}

    \begin{figure}[ht]
        \centering
        \includegraphics[scale = 0.6]{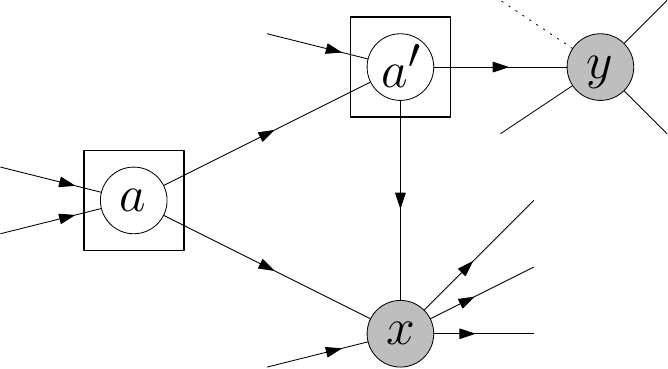}
        \caption{Illustration of the proof of \Cref{claim:6arc4claim}.}\label{fig:liaison_degre6_degre4}
    \end{figure}

    Let $x$ be a vertex of degree 6. By \Cref{lem:degree6balanced}, $\degout(x) = \degin(x) = 3$. Let $a, b, c$ be the three in-neighbors of $x$. By \Cref{lem:degree5}, $a,b,c$ have degree 4 or degree 6. We consider cases based on how many of these neighbors have degree equal to 4.

    If $\deg(a) = \deg(b) = \deg(c) = 4$ then by \Cref{claim:6arc4claim}, $x$ has three out-neighbors $a', b', c'$ of degree 4 such that $aa', bb', cc' \in E$. By removing $x,a,b,c,a',b',c'$ we lose at least 13 arcs. Moreover, for every acyclic set $A'$ of $G' = G-x-a-b-c-a'-b'-c'$, $A = A' + x + a' + b' + c'$ is an acyclic set in $G$. By applying \Cref{prop:reduction} with the triple $(7, 13, 4)$, we find that $G$ is not a minimal counterexample to \Cref{thm:main} (see \Cref{subfig:degree6a}).

    If $\deg(a) = \deg(b) = 4$ and $\deg(c) = 6$ then by \Cref{claim:6arc4claim}, $x$ has two out-neighbors $a', b'$ of degree 4 such that $aa', bb' \in E$. By removing $x,a,b,c,a',b'$ we lose at least 15 arcs. Indeed there are 7 arcs among the vertices mentioned, 7 half-arcs entering $a,b,c$ as well as 1 half-arc leaving $x$ which cannot be connected to any entering half-arc in the structure. Moreover, for every acyclic set $A'$ of $G' = G-x-a-b-c-a'-b'$, $A = A' + x + a' + b'$ is an acyclic set in $G$. By applying \Cref{prop:reduction} with the triple $(6, 15, 3)$, we find that $G$ is not a minimal counterexample to \Cref{thm:main} (see \Cref{subfig:degree6b}).

    If $\deg(a) = 4$ and $\deg(b) = \deg(c) = 6$ then by \Cref{claim:6arc4claim}, $x$ has an out-neighbor $a'$ of degree 4 such that $aa' \in E$. By removing $x,a,b,c,a'$ we lose at least 17 arcs. Indeed, there are 5 arcs among the cited vertices, 8 incoming half-arcs of $a,b,c$, as well as 2 outgoing half-arcs of $x$ that cannot be matched to any incoming half-arc of the structure. Moreover, by \Cref{lem:2OutNeighbors} applied to $a$, we cannot have both $ba \in E$ and $ca \in E$, so at least one outgoing half-arc of $b$ or $c$ is not matched to an incoming half-arc of the structure. Similarly, applying \Cref{lem:2OutNeighbors} to $a'$, we show that we cannot have both $a'b \in E$ and $a'c \in E$, so at least one outgoing half-arc of $a'$ is not matched to an incoming half-arc of the structure. Altogether, we thus remove at least 17 arcs. Moreover, for every acyclic set $A'$ of $G' = G-x-a-b-c-a'$, $A = A' + x + a'$ is an acyclic set in $G$. By applying \Cref{prop:reduction} with the triple $(5, 17, 2)$, we find that $G$ is not a minimal counterexample to \Cref{thm:main} (see \Cref{subfig:degree6c}).
    
    Thus $x$ has no in-neighbor of degree 4 and by symmetry it has no out-neighbor of degree 4.
\end{proof}

\begin{figure}[ht]
    \centering
    \begin{subfigure}[ht]{0.32\textwidth}
        \centering
        \includegraphics[scale = 0.5]{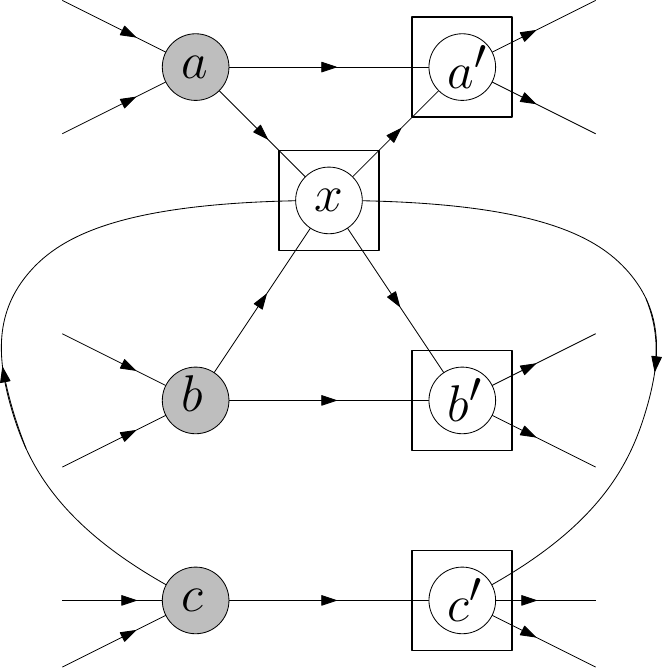}
        \caption{Case with three 4-in-neighbors.}\label{subfig:degree6a}
    \end{subfigure}
    \begin{subfigure}[ht]{0.32\textwidth}
        \centering
        \includegraphics[scale = 0.5]{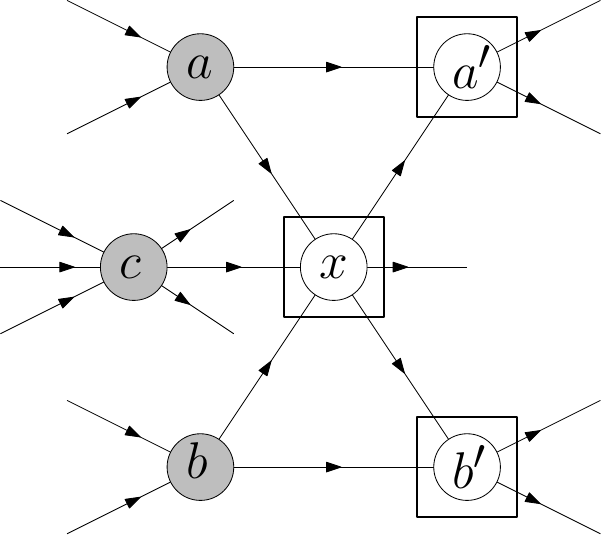}
        \caption{Case with two 4-in-neighbors.}\label{subfig:degree6b}
    \end{subfigure}
    \begin{subfigure}[ht]{0.32\textwidth}
        \centering
        \includegraphics[scale = 0.5]{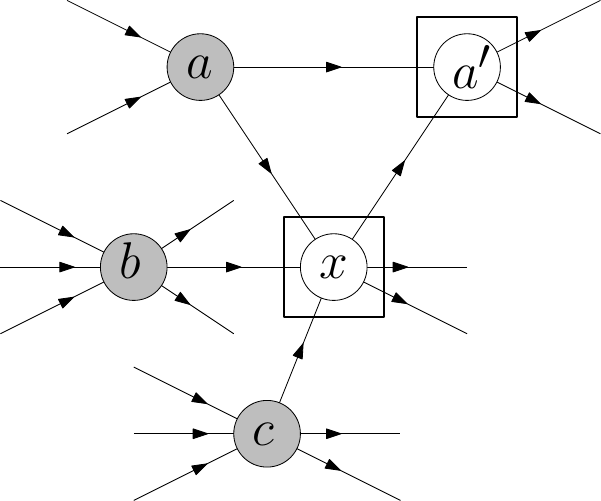}
        \caption{Case with one 4-in-neighbor.}\label{subfig:degree6c}
    \end{subfigure}
    \caption{Illustration of the proof of \Cref{lem:6arc4}.}
\end{figure}

\begin{lemma}\label{lem:3OutNeighbors}
    Let $x$ be a vertex of degree 6. The three in-neighbors (respectively out-neighbors) of $x$ form an oriented triangle.
\end{lemma}

\begin{proof}
    Let $a,b,c$ be the three in-neighbors of $x$. Suppose they do not form a triangle (oriented or non-oriented). Then removing $x,a,b,c$ removes at least 19 arcs. Moreover, for every acyclic set $A'$ of $G' = G-x-a-b-c$, $A = A' + x$ is an acyclic set of $G$. By applying \Cref{prop:reduction} with the triple $(4, 19, 1)$, we find that $G$ is not a minimal counterexample to \Cref{thm:main} (see \Cref{subfig:degree6trianglea}). Thus the three in-neighbors of $x$ form a triangle. 

    Suppose that $a,b,c$ form a non-oriented triangle. Without loss of generality, assume that $ba,ca,cb \in E$. Let $y$ be the third in-neighbor of $a$. Then removing $x,a,b,c,y$ removes at least 17 arcs. Moreover, for every acyclic set $A'$ of $G' = G-x-a-b-c-y$, $A = A' + x + a$ is an acyclic set of $G$. By applying \Cref{prop:reduction} with the triple $(5, 17, 2)$, we find that $G$ is not a minimal counterexample to \Cref{thm:main} (see \Cref{subfig:degree6triangleb}). Thus the three in-neighbors of $x$ form an oriented triangle. Similarly, the three out-neighbors of $x$ form an oriented triangle.
\end{proof}

\begin{figure}[ht]
    \centering
    \begin{subfigure}[ht]{0.4\textwidth}
        \centering
        \includegraphics[scale = 0.5]{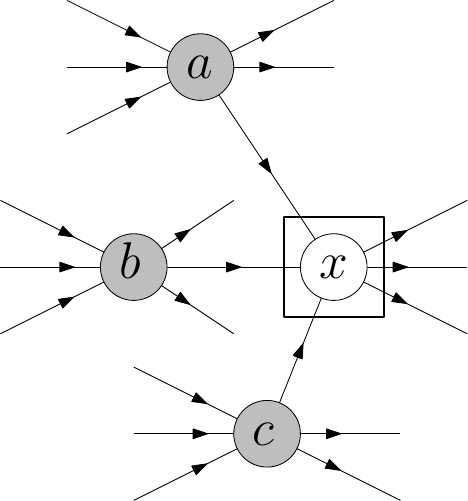}
        \caption{Case where $abc$ is not a triangle.}\label{subfig:degree6trianglea}
    \end{subfigure}
    \begin{subfigure}[ht]{0.58\textwidth}
        \centering
        \includegraphics[scale = 0.5]{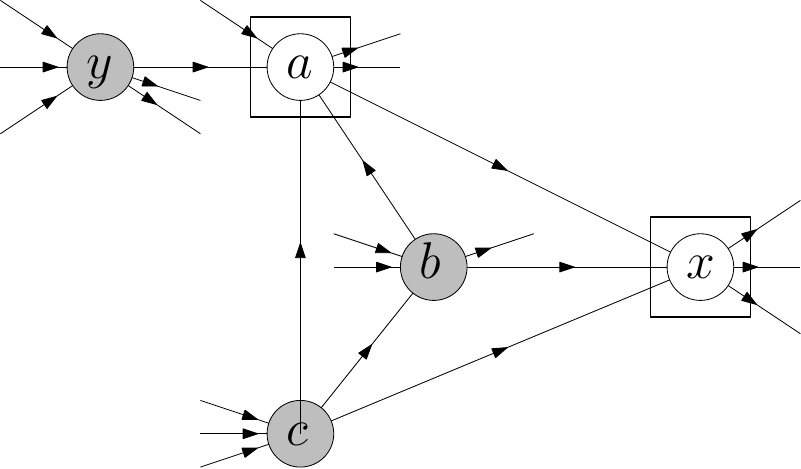}
        \caption{Case where $abc$ is a non-oriented triangle.}\label{subfig:degree6triangleb}
    \end{subfigure}
    \caption{Illustration of the proof of \Cref{lem:3OutNeighbors}.}
\end{figure}

\begin{lemma}\label{lem:degree6}
    $G$ contains no vertex of degree 6.
\end{lemma}

\begin{proof}
    Let $x_0$ be a vertex of degree 6. Let $x_1, x_2, x_4$ be its three out-neighbors and $x_3, x_5, x_6$ be its three in-neighbors. By \Cref{lem:3OutNeighbors}, $x_1, x_2, x_4$ form an oriented triangle and $x_3, x_5, x_6$ form an oriented triangle. Without loss of generality, assume that $x_1 x_2, x_2 x_4, x_4 x_1 \in E$ and $x_3 x_5, x_5 x_6, x_6 x_3 \in E$ (see \Cref{subfig:degree6withoutplanarityA} for an illustration).

    The third in-neighbor of $x_1$ (respectively $x_2$, $x_4$) forms an oriented triangle with $x_0$ and $x_4$ (respectively $x_1$, $x_2$). Since we already have $x_0 x_4 \in E$ (respectively $x_0 x_1 \in E$, $x_0 x_2 \in E$), it must necessarily be an in-neighbor of $x_0$. Without loss of generality, assume that the third in-neighbor of $x_1$ is $x_6$, so $x_6 x_1 \in E$. By \Cref{lem:3OutNeighbors}, the 3 in-neighbors of $x_1$ form an oriented triangle, so $x_4 x_6 \in E$, and the 3 out-neighbors of $x_6$ form an oriented triangle, so $x_1 x_3 \in E$ (see \Cref{subfig:degree6withoutplanarityB}).

    \begin{figure}[ht]
        \centering
        \begin{subfigure}[ht]{0.48\textwidth}
            \centering
            \includegraphics[scale = 0.5]{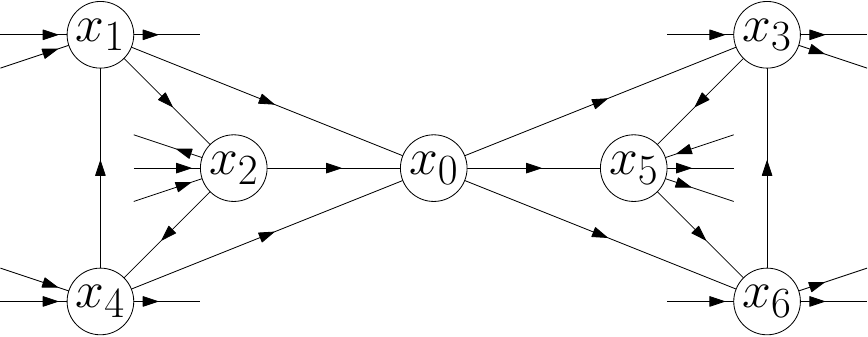}
            \caption{In- and out-neighbors of $x_0$.}\label{subfig:degree6withoutplanarityA}
        \end{subfigure}
        \begin{subfigure}[ht]{0.48\textwidth}
            \centering
            \includegraphics[scale = 0.5]{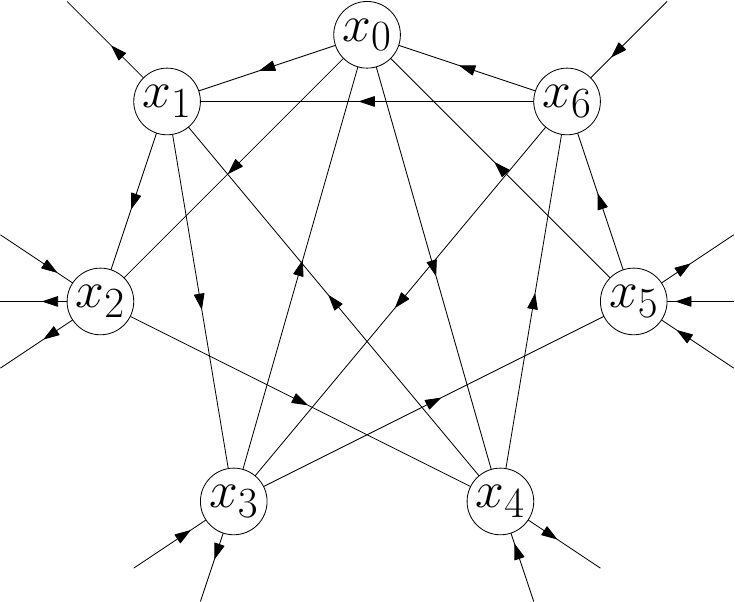}
            \caption{New arcs assuming $x_6 x_1 \in E$.}\label{subfig:degree6withoutplanarityB}
        \end{subfigure}
        \caption{Illustration of the arcs and vertices defined in the proof of \Cref{lem:degree6}}
    \end{figure}

    We know that the third in-neighbors of $x_2$ and $x_4$ are respectively $x_5$ and $x_3$. We then have two possibilities: either $x_5 x_2 \in E$ and $x_3 x_4 \in E$, or $x_3 x_2 \in E$ and $x_5 x_4 \in E$.

    \begin{claim}\label{claim:degre6sansplanariteclaim1}
        $x_5 x_2 \in E$ and $x_3 x_4 \in E$
    \end{claim}
    \begin{claimproof}[Proof of \Cref{claim:degre6sansplanariteclaim1}]
        Suppose that $x_3 x_2 \in E$ and $x_5 x_4 \in E$. Since the 3 out-neighbors of $x_3$ form an oriented triangle, we have $x_2 x_5 \in E$.

        Note that the last out-neighbor of $x_1$ is not some $x_i$. Indeed, the only $x_i$ with in-half arcs are $x_3, x_5, x_6$, and $x_3$ and $x_6$ are already adjacent to $x_1$. Finally, the arc $x_1 x_5$ cannot be in $E$ because otherwise the 3 out-neighbors of $x_1$ would not form an oriented triangle.

        We can reason similarly for all other vertices $x_i$ for $i \in \{2,3,4,5,6\}$ to justify that the last out-neighbors of $x_2, x_4$ and the last in-neighbors of $x_3, x_5, x_6$ are not some $x_i$. Thus removing $x_0, x_1, x_2, x_3, x_4, x_5, x_6$ removes at least 22 arcs. Moreover, for every acyclic set $A'$ of $G' = G-x_0-x_1-x_2-x_3-x_4-x_5-x_6$, $A = A' + x_3 + x_4 + x_6$ is an acyclic set of $G$. By applying \Cref{prop:reduction} with the triple $(7, 22, 3)$, we find that $G$ is not a minimal counterexample to \Cref{thm:main} (see \Cref{subfig:degree6withoutplanarityC}). Thus $x_5 x_2 \in E$ and $x_3 x_4 \in E$.
    \end{claimproof}

    By \Cref{claim:degre6sansplanariteclaim1}, $x_5 x_2 \in E$ and $x_3 x_4 \in E$. Since the 3 in-neighbors of $x_2$ (respectively $x_4$) form an oriented triangle, we have $x_1 x_5 \in E$ (respectively $x_2 x_3 \in E$). Since the 3 out-neighbors of $x_3$ (respectively $x_5$) form an oriented triangle, we have $x_4 x_5 \in E$ (respectively $x_2 x_6 \in E$).

    We then obtain a graph isomorphic to $P^7$ induced by the vertices $x_0, x_1, x_2, x_3, x_4, x_5, x_6$ (see \Cref{subfig:degree6withoutplanarityD}). This contradicts the fact that $h(G) = 0$ (\Cref{lem:hNul}), and we conclude that $G$ contains no vertex of degree 6.
\end{proof}

\begin{figure}[ht]
    \begin{subfigure}[ht]{0.48\textwidth}
        \centering
        \includegraphics[scale = 0.5]{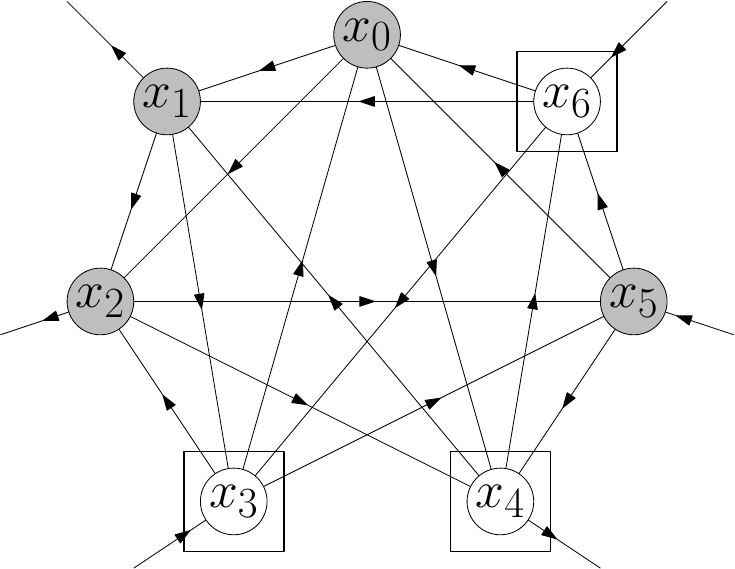}
        \caption{Case where $x_3 x_2 \in E$ and $x_5 x_4 \in E$.}\label{subfig:degree6withoutplanarityC}
    \end{subfigure}
    \begin{subfigure}[ht]{0.48\textwidth}
        \centering
        \includegraphics[scale = 0.5]{Figures/degree6finalD.pdf}
        \caption{Case where $x_5 x_2 \in E$ and $x_3 x_4 \in E$.}\label{subfig:degree6withoutplanarityD}
    \end{subfigure}
    \caption{Illustration of the proof of \Cref{lem:degree6}.}
\end{figure}

\subsection{End of the proof of \Cref{thm:main}}

From \Cref{lem:degree7,lem:degree1,lem:degree2,lem:degree3,lem:degree5,lem:degree6} we deduce that $G$ is 4-regular and thus $m = 2n$. To conclude we use the \Cref{thm:AGLYZ} recently obtained by Ai, Gutin, Liu, Yeo and Zhou~\cite{AGLYZ25} that states that every oriented graph $G$ with maximum degree at most 4 has a feedback vertex set of size at most $\frac{3n}{7}$.

Thus $a(G) \ge \frac{4n}{7} > \frac{5n}{9} = \frac{7n - m}{9}$, which contradicts the fact that $G$ is a counterexample to \Cref{thm:main}.

\section{Conclusion}
To establish the formula of \Cref{thm:main}, we started from the formula of \Cref{cor:planar}. Trying to relax the planarity assumption, the only counterexamples to the formula of \Cref{cor:planar} appear to be the graphs in $\mathcal{H}$. This motivated the addition of the parameter $h$, which counts the number of connected components of $G$ belonging to the family $\mathcal{H}$. In fact, with the planarity assumption, \Cref{cor:planar} can be proved more quickly by deducing from \Cref{lem:6arc4} the absence of degree 6 vertices in a minimal counterexample, observing that 6-regular planar graphs do not exist.

Note that the formula of \Cref{cor:planar} is tight: Knauer, Valicov and Wenger~\cite{KVW17} constructed an infinite family of planar oriented graphs $G$ with $\fv(G)=\frac{n-1}{2}$ and $m=\frac{5}{2}n-\frac{9}{2}$, which indeed yields $\fv(G)=\frac{2n+m}{9}$. More generally, for every $g \ge 3$, the authors of~\cite{KVW17} proposed a construction of infinite families of planar oriented graphs $G$ of digirth $g$ with $\fv(G)=\frac{n-1}{g-1}$.For $g \ge 12$, better constructions were recently obtained by the author together with Pinlou and Valicov~\cite{DPV26}: infinite families of planar oriented graphs $G$ of digirth $g$ with $\fv(G)=\frac{(g+2)n-2g}{g^2}$. Furthermore, they conjectured that these are the planar oriented graphs of digirth $g$ that maximize the ratio $\frac{\fv(G)}{n}$. Interestingly, all the above mentioned graphs (in~\cite{KVW17} and~\cite{DPV26}) satisfy $\fv(G)=\frac{2n+m}{3g}$, which suggests that the formula of \Cref{cor:planar} could be generalized to planar oriented graphs of digirth $g$ by replacing 9 with $3g$.

\begin{conjecture}\label{conj:3g}
    Every planar oriented graph $G$ of digirth $g$ satisfies $\fv(G) \le \frac{2n+m}{3g}$.
\end{conjecture}

We would like to notice that \Cref{conj:3g} combined with \Cref{thm:DPV26} would imply that $\fv(G) \le \frac{(g+2)n - 2g}{g^2}$ for planar oriented graphs of digirth $g \ge 3$ (the case $g = 3$ being the \Cref{cor:planar}), and this would be tight for $g \ge 12$ according to the constructions given in~\cite{DPV26}.

\bibliographystyle{plain}
\bibliography{biblio}

\end{document}